\documentclass[12pt]{amsart}
\usepackage{main}

\def\grad{\nabla}
\def\div{\operatorname{div}}

\def\hat{\widehat}
\def\tilde{\widetilde}

\def\inter{\text{int}}

\def\K{\mathcal K}

\def\cydot{\leavevmode\raise.4ex\hbox{.}}

\title[PAT and TAT with an uncertain wave speed]{Photoacoustic and thermoacoustic tomography with an uncertain wave speed}
\author{Lauri Oksanen}
\author{Gunther Uhlmann}
\address{University of Washington,
Department of Mathematics,
Box 354350,
Seattle, WA 98195-4350.}
\email{lauri.oksanen@math.washington.edu}
\email{gunther@math.washington.edu}
\date{\today}
\begin{document}

\begin{abstract}
We consider the mathematical model of photoacoustic and thermoacoustic tomography in media with
a variable sound speed. When the sound speed is known, 
the explicit reconstruction formula of \cite{Stefanov2009a} can be used.
We study how a modelling error in the sound speed affects the reconstruction 
formula and quantify the effect in terms of a stability estimate.
\end{abstract}

\maketitle

\section{Introduction}

Coupled-physics imaging methods, also called hybrid methods, attempt to
combine the high resolution of one imaging method with the high contrast
capabilities of another through a physical principle.  One important
medical imaging application is breast cancer detection. Ultrasound provides
a high, sub-millimeter resolution, but suffers from low contrast.  On the
other hand, many tumors absorb much more energy of electromagnetic waves,
in some specific energy bands, than healthy cells. 

Photoacoustic tomography (PAT) \cite{Wang2009} consists of sending relatively harmless optical
radiation into tissues that causes heating, with increases of the
temperature in the millikelvin range, which results in the generation of
propagating ultrasound waves. This is the  photo-acoustic effect. 
The inverse problem then consists of
reconstructing the optical properties of the tissue. In Thermoacoustic
tomography (TAT), see e.g. \cite{Kruger1999}, low frequency microwaves,
with wavelengths on the order of $1m$, are sent into the medium. The
rationale for using the latter frequencies is that they are less absorbed
than optical frequencies.  PAT, TAT and
other couple-physics imaging techniques offer potential breakthroughs in
the clinical application of multi-wave methods to early detection of
cancer, functional imaging, and molecular imaging among others \cite{Xu2006, Wang2009}.

There are two steps in PAT and TAT. The first one is to solve the inverse source
problem for the wave equation. The source  measures the absorbed radiation and
varies from tissue to tissue. Once this is solved the second step is to determine the
optical or electromagnetic parameters from the internal information obtained in the first step. 
Here we are concerned with the first step that we now describe.


Let $u$ solve the initial value problem
\begin{equation}   \label{eq_wave}
\left\{
\begin{array}{rcll}
(\partial_t^2 -c^2\Delta)u &=&0 &  \mbox{in $(0,T)\times \R^n$},\\u|_{t=0} &=& f,\\ \quad \partial_t u|_{t=0}& =&0,
\end{array}\right.\end{equation}
where $T>0$ is fixed and $c$ is a smooth and strictly positive function on $\R^n$.
Let us assume that the source $f$ is supported in $M$, where $M\subset \R^n$ is
some bounded domain with smooth boundary. The measurements are modelled by the operator
\begin{align}	\label{def_measurement}
\Lambda_c f : = u|_{[0,T]\times \p M}.
\end{align}
The inverse problem is to reconstruct the unknown $f$ and $c$ from $\Lambda_c$. 

Most of the available results assume that the sound speed $c$ is known.
Concerning the constant speed case we refer to the survey paper \cite{Kuchment2008} and
the references therein.
In practice, there are many cases
when the constant sound speed model is inaccurate. For instance in breast imaging,
the different components of the breast, such as the glandular tissues, stromal tissues, cancerous
tissues and other fatty issues, have different acoustic properties. The variations between their
acoustic speeds can be as great as 10 percent \cite{Jin2006}. 
The case of a variable sound speed was thoroughly investigated in \cite{Stefanov2009a} and a reconstruction method for the source was proposed.
A numerical algorithm based on this method was developed in \cite{Qian2011} and compared with the standard time reversal method used for instance in \cite{Hristova2009, Hristova2008}.
The case of a discontinuous sound speed was considered in \cite{Stefanov2011a} which arises, for instance, in brain imaging. 

The paper \cite{Stefanov2011} deals with the case where the source $f$ is known and recovers the sound speed $c$ with one measurement under  the geometric assumption
that the domain is foliated by strictly convex hypersurfaces with respect to the Riemannian metric $g=\frac{1}{c^2} dx^2$ where $dx^2$ denotes the Euclidean
metric. In practice the wave speed is known only up to some uncertainty
and we would like to reconstruct both $c$ and $f$ given $\Lambda_c f$.  This problem is linear in $f$ and non-linear in $c.$ It is shown in \cite{Stefanov2012} that the linearized problem of recovering both parameters is unstable.



In this paper, we assume that our best guess for the wave speed $c_0$ is
close to the true wave speed $c$.
We will study how the modelling error $c - c_0$
affects the reconstruction of $f$ by the 
modified time-reversal method introduced in \cite{Stefanov2009a}.

\section{Statement of the results}

Let us begin by describing the reconstruction method introduced in \cite{Stefanov2009a}. 
To define the operator giving the reconstruction, 
we consider the time reversed wave equation on $M$,
\begin{align}
\label{eq_wave_reversal}
&\p_t^2 v - c^2 \Delta v = 0, & \text{in $(0,T) \times M$},
\\\nonumber& v|_{x \in \p M} = h, & \text{in $(0,T) \times \p M$},
\\\nonumber& v|_{t=T} = \Delta^{-1} h|_{t=T}, \quad \p_t v|_{t=T}= 0,  & \text{in $M$},
\end{align}
where $\Delta^{-1} h|_{t=T}$ is the solution of the Dirichlet problem on $M$,
\begin{align*}
\Delta \phi = 0\ \text{in $M$} \quad \text{and} \quad \phi = h|_{t=T}\ \text{on $\p M$}.
\end{align*}
Moreover, we define the operators
\begin{align*}
A_c h = v|_{t = 0} \quad \text{and} \quad K_c = 1 - A_c \Lambda_{c},
\end{align*}
where $v$ is the solution of (\ref{eq_wave_reversal}).
The method of \cite{Stefanov2009a} gives a reconstruction by the formula
\begin{align*}
f = R_c \Lambda_{c} f, \quad \text{where} \quad R_c := \sum_{m=0}^\infty K_c^m A_c.
\end{align*}

Let $c_0 \in C^\infty(\R^n)$ be strictly positive and let us consider the
reconstruction operator $R_{c_0}$ corresponding to $c_0$.
If $c_0$ is close to $c$, we expect the reconstruction 
\begin{align}
\label{def_reconstruction}
\tilde f = R_{c_0} \Lambda_{c} f.
\end{align}
to be close to $f$.
Indeed, in this paper we show that the difference $\tilde f - f$
is bounded by suitable norms of the measurement $\Lambda_c f$
and the modelling error $c - c_0$.

As mentioned above, the closely related problem to find $c$ given the pair $(\Lambda_c f, f)$ 
was considered in \cite{Stefanov2011}, and the problem was shown to be stable 
under the geometric assumption
that $M$ is foliated by strictly convex hypersurfaces with respect to the Riemannian metric $c^{-2} dx^2$.
We will impose a similar convexity condition below. 
Furthermore, we will assume that $\p M$ is strictly convex with respect to the Riemannian metric $c^{-2} dx^2$.
If this is the case,  
then the operator $\Lambda_{c}$ has the following regularity property
\begin{align*}
\Lambda_{c} : H_0^1(\K) \to H^1((0,T) \times \p M),
\end{align*}
where $\K \subset M^\inter$ is compact, see e.g. \cite{Stefanov2009a}.
Without the convexity assumption, a loss of smoothness up to degree $1/4$ is possible \cite{Tataru1998}.
We describe this phenomenon in more detail in the appendix below. 

The operator $K_c$ is compact if the state of the wave equation (\ref{eq_wave})
is smooth at $t=T$ irrespective of $f$. 
This is guaranteed if all unit speed geodesics starting from $M$ at $t=0$
exit $M$ before $t=T$ and do not re-enter before $t=T$.
To avoid technicalities related to re-entering singularities we assume below that 
$M$ is convex also in the sense that no geodesic exiting $M$ re-enters $M$.

\begin{theorem}
\label{thm_main}
Let $c_0 \in C^\infty(\R^n)$ be strictly positive and suppose that 
the Riemannian manifold $(M, c_0^{-2} dx^2)$ has a strictly convex function with no critical points
and that $M$ is strictly convex with respect to the Riemannian metric $c_0^{-2} dx^2$.
Let $\K \subset M^\inter$ be compact, let $C_c, C_f > 0$ and consider the functions
$f \in H^3(\R^n)$ and $c \in C^\infty(\R^n)$ satisfying
\begin{align*}
\norm{f}_{H^3(\K)} \le C_f, \quad \norm{c}_{C^2(\K)} \le C_c, \quad
\supp(f) \subset \K, \quad c = c_0\ \text{in $\R^n \setminus \K$}.
\end{align*}
There are $\epsilon_c, T, C > 0$ such that if $c$ satisfies also 
\begin{align*}
\norm{c - c_0}_{C^1(\K)} \le \epsilon_c
\end{align*}
then
\begin{align}
\label{main_stability}
\norm{(R_{c_0} - R_c) \Lambda_c f}_{H^1(M)} \le C \norm{c - c_0}_{L^\infty(M)} \norm{\Lambda_c f}_{H^1((0,T) \times \p M)}^{1/2}.
\end{align}
\end{theorem}

We emphasize that $(R_{c_0} - R_c) \Lambda_c f = \tilde f - f$, whence (\ref{main_stability})
gives a bound for the reconstruction error.


\section{The proof}

The existence of a strictly convex function on $(M, c_0^{-2} dx^2)$ implies that the wave equation
\begin{align*}
&\p_t^2 v - c_0^2 \Delta v = 0, & \text{in $(0,T) \times M$},
\end{align*}
is exactly controllable for large $T > 0$, see e.g. \cite{Yao1999}. 
This again implies that $(M, c_0^{-2} dx^2)$ is non-trapping, see e.g. \cite{Bardos1992}.
Thus $\sum_{m =0}^\infty K_{c_0}^m$ is a bounded operator on $H_0^1(M)$ by \cite{Stefanov2009a}.
Let us begin by showing that the reconstruction (\ref{def_reconstruction}) is well defined.
This is guaranteed by the next lemma.

\begin{lemma}
Let $f \in H_0^1(M) \cap H^2(M)$. Then $A_{c_0} \Lambda_c f \in H_0^1(M)$.
\end{lemma}
\begin{proof}
Notice that the solution $u$ of (\ref{eq_wave}) is in $C([0, T]; H^2(M))$. 
Let us consider the solution $v$ of 
\begin{align*}
&\p_t^2 v - c_0^2 \Delta v = 0, & \text{in $(0,T) \times M$},
\\\nonumber& v|_{x \in \p M} = \Lambda_c f, & \text{in $(0,T) \times \p M$},
\\\nonumber& v|_{t=T} = \phi, \quad \p_t v|_{t=T}= 0,  & \text{in $M$},
\end{align*}
where $\phi = \Delta^{-1} \Lambda_c f|_{t=T}$.
The difference $w = u - v$ satisfies 
\begin{align*}
&\p_t^2 w - c_0^2 \Delta w = (c^2 - c_0^2) \Delta u,  & \text{in $(0,T) \times M$},
\\\nonumber& w|_{x \in M} = 0,  & \text{in $(0,T) \times \p M$}, 
\\\nonumber& w|_{t=T} = u|_{t=T} - \phi,
\quad \p_t w|_{t=T}= \p_t u|_{t=T},  & \text{in $M$}.
\end{align*}
Notice that $w|_{t=T} = 0$ on $\p M$ and $\Delta u \in C([0, T]; L^2(M))$.
Thus we have $w \in C([0, T]; H_0^1(M))$, 
and $A_{c_0} \Lambda_c f = v|_{t =0} = f - w|_{t = 0} \in H_0^1(M)$.
\end{proof}

Let us recall that $A_{c_0} : H^1((0,T) \times \p M) \to H^1(M)$ is bounded \cite{Lasiecka1986},
and that $f = R_{c_0} \Lambda_{c_0} f$ \cite{Stefanov2009a}.
Thus there is $C > 0$ depending only on $c_0$ and $M$ such that
\begin{align*}
\norm{\tilde f - f}_{H^1(M)} &= \norm{R_{c_0} (\Lambda_c f - \Lambda_{c_0} f)}_{H^1(M)}
\\&\le C \norm{\Lambda_c f - \Lambda_{c_0} f}_{H^1((0,T) \times \p M)}.
\end{align*}
Notice that $\Lambda_c f - \Lambda_{c_0} f = w|_{(0,T) \times \p M}$,
where $w$ is the solution of the wave equation,
\begin{align*}
&\p_t^2 w - c_0^2 \Delta w = (c^2 - c_0^2) \Delta u,  & \text{in $(0,\infty) \times \R^n$},
\\\nonumber& w|_{t=0} = 0, \quad \p_t w|_{t=0}= 0,  & \text{in $\R^n$},
\end{align*}
and $u$ is the solution of (\ref{eq_wave}).

By Theorem \ref{thm_regularity} in the appendix below, 
there is a constant $C > 0$ depending only on $c_0$, $T$ and $M$
such that 
\begin{align*}
\norm{w|_{(0,T) \times \p M}}_{H^1((0,T) \times \p M)} 
&\le 
C \norm{(c^2 - c_0^2) \Delta u}_{L^2((0,T) \times \K)}
\\&\le 
C \norm{c^2 - c_0^2}_{L^\infty(\K)} \norm{\Delta u}_{L^2((0,T) \times \K)}.
\end{align*}
By applying the energy estimates, see e.g. \cite[p. 153]{Ladyzhenskaya1985}, on $\p_t u$ we see that there is a constant $C > 0$ depending only on $T$ and the bounds
\begin{align}
\label{energy_bounds}
\epsilon_0 \le c(x), \quad \norm{c}_{C^1(\R^n)} \le C_0,
\end{align}
such that 
\begin{align*}
\norm{\Delta u}_{L^2((0,T) \times \K)} \le C \norm{f}_{H^2(\K)}
\le C \norm{f}_{H^{1}(\K)}^{1/2} \norm{f}_{H^{3}(\K)}^{1/2}.
\end{align*}
Thus there is $C > 0$ depending only on $C_f$, $c_0$, $T$, $M$ and the bounds (\ref{energy_bounds}) 
such that
\begin{align*}
\norm{\Lambda_c f - \Lambda_{c_0} f}_{H^1((0,T) \times \p M)}
\le C \norm{c^2 - c_0^2}_{L^\infty(\K)} \norm{f}_{H^{1}(\K)}^{1/2}.
\end{align*}

By \cite{Stefanov2009a} there is a constant $C(c) > 0$
depending on $c$ such that 
\begin{align}
\label{stability_c0}
\norm{f}_{H^{1}(M)} \le C(c) \norm{\Lambda_c f}_{H^1((0,T) \times \p M)}.
\end{align}
Theorem \ref{thm_main} follows after we show that $C(c)$ is uniformly bounded in a neighborhood of $c_0$.
We do not know how to prove this simply by perturbing the estimate (\ref{stability_c0}).
However, the uniform boundedness follows from a modification of the Carleman estimate in \cite{Liu2012}.
We will prove the modified estimate in the next section
under the geometric assumptions formulated in Theorem \ref{thm_main},
see Corollary \ref{cor_stable_observability} below.


\section{A Carleman estimate with explicit constants}

As the proof of the Carleman estimate below is of geometric nature, we will 
consider the wave equation 
\begin{align*}
&\p_t^2 u - \Delta_{g,\mu} u = 0, &\text{in $(0, T) \times M$}
\end{align*}
on a smooth compact Riemannian manifold $(M,g)$
with boundary. Here $\Delta_{g,\mu}$ is the weighted Laplace-Beltrami operator,
\begin{align*}
\Delta_{g,\mu} u = \mu^{-1} \div_g (\mu \grad_g u),
\end{align*}
where $\mu \in C^\infty(M)$ is strictly positive
and $\div_g$ and $\grad_g$ denote the divergence and the gradient 
with respect to the metric tensor $g$.
In order to prove Theorem \ref{thm_main}, 
we apply the results below to $g = c(x)^{-2} dx^2$
and $\mu(x) = c(x)^{n-2}$. Then $\Delta_{g,\mu} = c(x)^2 \Delta$, where $\Delta$ is the Euclidean Laplacian. 

We denote by $|\cdot|_g$, $(\cdot, \cdot)_g$, $dV_g$, $dS_g$, $\nu_g$ and $D^2_g$
the norm, the inner product,
the volume and the surface measures, the exterior unit normal vector
and the Hessian with respect to $g$.
Moreover, we write $\div_{g,\mu} X := \mu^{-1} \div_g (\mu X)$ for a vector field $X$,
and will omit writing the subscript $g$ when considering a fixed Riemannian metric tensor.

Let us recall the pointwise Carleman estimate from \cite{Liu2012}.

\begin{theorem}
\label{th_carleman_ineq}
Let $\ell \in C^3(M)$ and $\rho > 0$ satisfy
\begin{align}
\label{carleman_ineq_rho}
D^2 \ell (X, X) \ge \rho |X|^2, \quad X \in T_x M,\ x \in M.
\end{align}
Let $\tau > 0$ and $u \in C^2(\R \times M)$.
We define
\begin{align*}
&v := e^{\tau \ell} u, \quad 
\vartheta := \tau ((\Delta_\mu \ell - \rho) v + 2 (\grad v, \grad \ell)),
\\&
Y := \tau ((\p_t v)^2 + |\grad v|^2 - (\tau \rho - \tau^2 |\nabla \ell|^2) v^2) \grad \ell.
\end{align*}
Then
\begin{align*}
&e^{2 \tau\ell} (\p_t^2 u - \Delta_\mu u)^2/2 
- \p_t (\vartheta\p_t v) + \div_\mu(\vartheta \grad v) + \div_\mu Y
\\&\quad\ge
e^{2 \tau\ell}(\rho \tau - 1)((\p_t u)^2 + |\nabla u|^2)/2
+ e^{2 \tau\ell}(2 \rho |\grad \ell|^2 \tau - C_1) \tau^2 u^2,
\end{align*}
where $C_1 = \rho^2 + \max_{x \in M} |\grad (\Delta_\mu \ell(x))|^2$.
\end{theorem}

\begin{lemma}
\label{lem_boundary_terms}
Let $T > 0$, $u \in C^2([0, T] \times M)$,
$\tau > 1$, $\rho > 0$ and let $\ell \in C^3(M)$.
We define $v$, $\vartheta$ and $Y$ as in Theorem \ref{th_carleman_ineq}.
Moreover, we write $dm := \mu dV$ and $dn := \mu dS$.
Then 
\begin{align*}
&\int_0^T \int_M \div_\mu(\vartheta \grad v) + \div_\mu Y\, dm dt 
\\&\quad\le 
(6C_2 + C_3)  e^{B_\ell \tau} \tau \int_0^T \int_{\p M} (\p_t u)^2 + |\grad u|^2 dn dt
\\&\quad\quad+
 (7 C_2^3 \tau^2 + 5 C_2 C_3 \tau) e^{B_\ell \tau} \tau \int_0^T \int_{\p M}  u^2 dn dt,
\end{align*}
where 
\begin{align*}
C_2 = \max_{x \in M} (|\grad \ell(x)| + 1),
\quad 
C_3 = \max_{x \in M} (\rho + |\Delta_\mu \ell|)/2,
\end{align*}
and $B_\ell = 2\max_{x \in M} \ell(x)$. Moreover,
\begin{align*}
&\int_M |\vartheta \p_t v| dm 
\\&\quad\le 
C_F (2 C_2^2 \tau + C_3) e^{B_\ell \tau} \tau 
\ll( \int_M (\p_t u)^2 + |\grad u|^2 dm + \int_{\p M} u^2 dn \rr),
\end{align*}
where $C_F \ge 1$ is a constant satisfying the
Friedrichs' inequality 
\begin{align}
\label{ineq_Friedrichs}
\int_M \phi^2 dm \le C_F \ll( \int_M |\grad \phi|^2 dm + \int_{\p M} \phi^2 dn \rr), \quad \phi \in C^\infty(M).
\end{align}
\end{lemma}
\begin{proof}
We have
\begin{align*}
\grad v &= e^{\tau \ell} \ll( \grad u + \tau u \grad \ell \rr), 
\\
\vartheta &= \tau e^{\tau \ell} \ll( (\Delta_\mu \ell - \rho + 2\tau |\grad \ell|^2)u + 2(\grad u, \grad \ell) \rr),
\\
Y &= \tau e^{2 \tau \ell} \ll( (\p_t u)^2 + |\grad u|^2 + 2 \tau u (\grad u, \grad \ell) + (2\tau^2 |\grad \ell|^2 - \tau \rho) u^2 \rr)\grad \ell.
\end{align*}
Thus 
\begin{align*}
&\tau^{-1} e^{-2 \tau \ell} (\vartheta (\grad v, \nu) + (Y, \nu))
\\&\quad= 
[ (\p_t u)^2 + |\grad u|^2 + 4 \tau u (\grad u, \grad \ell) 
\\&\quad\quad\quad+
(4 \tau^2 |\grad \ell|^2 - 2 \tau \rho + \tau \Delta_\mu \ell) u^2 ] (\grad \ell, \nu)
\\&\quad\quad+
\ll( (\Delta_\mu \ell - \rho + 2 \tau |\grad \ell|^2 ) u + 2 (\grad u, \grad \ell) \rr) (\grad u, \nu)
\end{align*}
We estimate the cross terms as follows
\begin{align*}
4 \tau |u (\grad u, \grad \ell)|
&\le 
2\tau^2 |\grad \ell|^2 u^2 + 2 |\grad u|^2,
\\
|\Delta_\mu \ell - \rho| |u (\grad u, \nu)|
&\le 
|\Delta_\mu \ell - \rho|/2\ u^2 + |\Delta_\mu \ell - \rho|/2\ |\grad u|^2,
\\
2 \tau |\grad \ell|^2 |u (\grad u, \nu)|
&\le 
\tau^2 |\grad \ell|^3 u^2 + |\grad \ell| |\grad u|^2,
\end{align*}
and get
\begin{align*}
&\tau^{-1} e^{-2 \tau \ell} |\vartheta (\grad v, \nu) + (Y, \nu)|
\\&\quad\le
|\grad \ell|(\p_t u)^2 + (6|\grad \ell| + (\rho + |\Delta_\mu \ell|)/2)|\grad u|^2 
\\&\quad\quad+
\ll(7 \tau^2 |\grad \ell|^3 + (4\tau |\grad \ell| + 1) (\rho + |\Delta_\mu \ell|)/2 \rr) u^2.
\end{align*}
The first claim follows from the divergence theorem.

For the second claimed inequality, notice that 
\begin{align*}
&\tau^{-1} e^{-2\tau \ell} |\vartheta \p_t v| = |((\Delta_\mu \ell - \rho) u + 2 (\grad u + \tau u \grad \ell, \grad \ell)) \p_t u|
\\&\quad\le
|(\Delta_\mu \ell - \rho) + 2 \tau |\grad \ell|^2| |u \p_t u| 
+ 2|\grad \ell| |\grad u| |\p_t u|
\\&\quad\le
|(\Delta_\mu \ell - \rho)/2 + \tau |\grad \ell|^2|(u^2 + (\p_t u)^2)
+ |\grad \ell|(|\grad u|^2 + (\p_t u)^2).
\end{align*}
Hence
\begin{align*}
&e^{-B_\ell \tau} \tau^{-1} \int_M |\vartheta \p_t v| dm 
\\&\quad\le 
(\max |\Delta_\mu \ell - \rho|/2 + \tau \max |\grad \ell|^2) \ll(\int_M u^2 dm + \int_M (\p_t u)^2 dm \rr)
\\&\quad\quad
+ \max |\grad \ell| \int_M |\grad u|^2 + (\p_t u)^2 dm,
\end{align*}
and the second claimed inequality follows from (\ref{ineq_Friedrichs}) with $\phi = u$.
\end{proof}

\begin{remark}
\label{rem_energy}
Let $u \in C^2([0, \infty) \times M)$ be a solution of 
\begin{align}
\label{eq_wave_control}
&\p_t^2 u(t,x) - \Delta_\mu u(t, x) = 0, & (t,x) \in (0,\infty) \times M,
\end{align}
Then the energy,
\begin{align*}
E(t) := \int_M (\p_t u(t))^2 + |\grad u(t)|^2 dm,
\end{align*}
satisfies
\begin{align*}
E(t) = E(0) + 2\int_0^t \int_{\p M} \p_\nu u(s)\ \p_t u(s)\ dn ds.
\end{align*}
\end{remark}

\begin{theorem}[Observability inequality]
\label{thm_obs_ineq}
Suppose that there is a strictly convex function $\ell \in C^3(M)$ with no critical points.
Let $\rho, r > 0$ satisfy 
\begin{align}
\label{conditions_for_ell}
D^2 \ell (X, X) > \rho |X|^2, \quad |\grad \ell(x)| > r, 
\end{align}
for all $X \in T_x M$ and $x \in M$.
We write $\beta_\ell = 2 \min_{x \in M} \ell(x)$ and suppose that 
\begin{align}
\label{def_T}
T > 2 C_F (2 C_2^2 \tau + C_3) e^{(B_\ell - \beta_\ell) \tau} \tau ,
\quad \text{where }
\tau = \max \ll(\frac{3}{\rho}, \frac{C_1}{2 \rho r^2},1 \rr).
\end{align}
Let $u \in C^2([0, T] \times M)$ be a solution of 
(\ref{eq_wave_control}).
Then
\begin{align*}
E(0) \le C \int_0^T \int_{\p M} (\p_t u)^2 + |\grad u|^2 + u^2 dn dt,
\end{align*}
where $C$ depends only on $C_2$, $C_3$, $\beta_\ell$, $B_\ell$, $\tau$ and $T$.
\end{theorem}
\begin{proof}
We will integrate the inequality of Theorem \ref{th_carleman_ineq}.
Notice that 
\begin{align*}
\rho \tau - 1 \ge 2 \quad \text{and} \quad 2 \rho |\grad \ell|^2 \tau - C_1 \ge 0,
\end{align*}
whence
\begin{align}
\label{ineq_prelim_energy_boundary}
&e^{\beta_\ell \tau}
\int_0^T \int_M (\p_t u)^2 + |\nabla u|^2 dm dt
\\\nonumber&\quad\le 
\int_0^T \int_M 
-\p_t (\vartheta\p_t v) + \div_\mu(\vartheta \grad v) + \div_\mu Y dm dt.
\end{align}
We estimate the left-hand side of (\ref{ineq_prelim_energy_boundary}) from below using Remark \ref{rem_energy},
\begin{align*}
&e^{\beta_\ell \tau}\int_0^T \int_M (\p_t u)^2 + |\nabla u|^2 dm dt
\\&\quad\ge
e^{\beta_\ell \tau} T E(0) - 
e^{\beta_\ell \tau} T 
\ll(\norm{\p_\nu u}_{L^2((0,T) \times \p M)}^2 + \norm{\p_t u}_{L^2((0,T) \times \p M)}^2 \rr).
\end{align*}
To estimate the right-hand side of (\ref{ineq_prelim_energy_boundary}) from above,
we notice that 
\begin{align*}
&C_0^{-1} \ll| \int_0^T \int_M \p_t (\vartheta\p_t v) dm dt \rr|
\\&\quad\le E(0) + E(T) + \int_{\p M} u(0)^2 + u(T)^2 dn 
\\&\quad\le 2 E(0) + \norm{\p_\nu u}_{L^2((0,T) \times \p M)}^2 + \norm{\p_t u}_{L^2((0,T) \times \p M)}^2 
\\&\quad\quad+ \int_{\p M} u(0)^2 + u(T)^2 dn,
\end{align*}
where $C_0 = C_F (2 C_2^2 \tau + C_3) e^{B_\ell \tau} \tau $.
Hence
\begin{align*}
&(e^{\beta_\ell \tau} T - 2C_0 ) E(0) 
\\&\quad\le
C_0 \ll(\norm{\p_\nu u}_{L^2((0,T) \times \p M)}^2 + 2C_{Tr} \norm{u}_{H^1((0,T) \times \p M)}^2\rr)
\\&\quad\quad+
(6C_2 + C_3)  e^{B_\ell \tau} \tau \ll( \norm{\p_t u}_{L^2((0,T) \times \p M)}^2 + \norm{\grad u}_{L^2((0,T) \times \p M)}^2 \rr)
\\&\quad\quad+
 (7 C_2^3 \tau^2 + 5 C_2 C_3 \tau) e^{B_\ell \tau} \tau \norm{u}_{L^2((0,T) \times \p M)}^2
\\&\quad\quad+
e^{\beta_\ell \tau} T 
\ll(\norm{\p_\nu u}_{L^2((0,T) \times \p M)}^2 + \norm{\p_t u}_{L^2((0,T) \times \p M)}^2 \rr),
\end{align*}
where $C_{Tr} \ge 1$ is a constant satisfying the
trace inequality 
\begin{align*}
\norm{\phi(0)}_{L^2(\p M)}^2 + \norm{\phi(T)}_{L^2(\p M)}^2 \le C_{Tr} 
\norm{\phi}_{H^1((0,T) \times \p M)}^2.
\end{align*}
\end{proof}

We recall the following trace regularity result by Lasiecka, Lions and Triggiani \cite{Lasiecka1986}.

\begin{theorem}
Let $f \in C_0^\infty(M)$. There is $C > 0$ depending only on $T$, $M$ and $c|_{\R^n \setminus M}$
such that 
\begin{align*}
\norm{\p_\nu u}_{L^2((0,T) \times \p M)} \le C \norm{\Lambda_c f}_{H^1((0,T) \times \p M)},
\end{align*}
where $u$ is the solution of (\ref{eq_wave}).
\end{theorem}
\begin{proof}
Apply \cite[Th. 2.1]{Lasiecka1986} in the domain $B(0, R) \setminus M$,
where $R > 0$ is large enough so that $u(t, x)$ vanish for $t \in [0,T]$
and $x \in \p B(0,R)$.
\end{proof}

\begin{corollary}[Stable observability]
\label{cor_stable_observability}
Suppose that there is a strictly convex function $\ell \in C^3(M)$ with no critical points
on the Riemannian manifold $(M, c_0^{-2} dx^2)$.
Let $U_0$ be a bounded $C^2$ neighborhood of $c_0$.
Then there is a $C^1$ neighborhood $U$ of $c_0$
and constants $C, T > 0$ satisfying the following:
for all $c \in U_0 \cap U$ 
the solution of the wave equation (\ref{eq_wave})
satisfies the observability inequality
\begin{align*}
\norm{f}_{H^1_0(M)} \le C \norm{\Lambda_c f}_{H^1((0, T) \times \p M)}.
\end{align*}
\end{corollary}
\begin{proof}
As in \cite{Liu2012}, we see that (\ref{conditions_for_ell})
holds for the Riemannian manifold $(M, c^{-2} dx^2)$
if $c \in U$ and $U$ is small enough,
and also that the constants 
$C_1$, $C_2$, $C_3$, $\beta_\ell$, $B_\ell$, $\tau$ and $T$
for $(M, c^{-2} dx^2)$ stay bounded when $c$ belongs to $U_0 \cap U$.
\end{proof}

\section*{Appendix: on the regularity of traces}

Let us denote by $\pi_0$ the projection on $T^*(\{0\} \times \K)$
from the subset of $T^* \R^{1+n}$ lying on $\{0\} \times \K$,
and let $\pi_{\p M}$ be the analogous projection with respect to the set $(0,T) \times \p M$.
Furthermore, let us denote by $\gamma_{x,\xi}$ the geodesic satisfying $\gamma(0) = x$ and $\dot \gamma(0) = \xi$, where $\dot \gamma$ is the differential of $\gamma$, and define 
\begin{align*}
\tau(x,\xi) := \min\{t > 0;\ \gamma_{x,\xi}(t) \in \p M \}.
\end{align*}

The measurement operator 
\begin{align*}
\Lambda_c : C_0^\infty(\K) \to C^\infty((0,T) \times \p M) 
\end{align*}
is the sum of two zeroth order Fourier integral operators with canonical relations 
$C_+$ and $C_-$,
and the union $C_+ \cup C_-$ consist of the points 
\begin{align*}
(\pi_0(s, y, \sigma, \eta), \pi_{\p M}(t, x, \tau, \xi))
\end{align*}
such that $(t, x, \tau, \xi) \in T^* \R^{1+n}$ and $(s, y, \sigma, \eta) \in T^* \R^{1+n}$
lie on the same null bicharacteristic of $p(x, \tau, \xi) := \tau^2 - c(x)^2|\xi|^2$,
$x \in \p M$, $y \in \K$ and $s = 0$. 
Moreover, in a neighborhood of such a point $(y_0, \eta_0) \in T^* \K$ that 
\begin{align} \label{non_tangential}
\dot \gamma_{y_0, \eta_0}(\tau(y_0, \eta_0)) \notin T^*(\p M),
\end{align}
the canonical relation $C_+$ is given by the graph of the map
\def\t{\intercal}
\begin{align*}
(y, \eta) \mapsto (\tau(y, \eta), \gamma_{y, \eta}(\tau(y, \eta)), |\eta|, \dot \gamma_{y, \eta}^\t(\tau(y, \eta))),
\end{align*}
where $\dot \gamma^\t$ stands for the tangential projection on $T^*(\p M)$.
An analogous statement holds for $C_-$ with $\eta$ replaced by $-\eta$, see e.g. \cite{Stefanov2009a}
for more details. In particular, when restricted near $(y_0, \eta_0)$ satisfying (\ref{non_tangential})
the measurement operator is continuous from $H^1$ to $H^1$ \cite[Cor. 25.3.2]{Hormander1985a}.

The case of tangential intersection, that is, the case when (\ref{non_tangential}) does not hold,
is more delicate and in general $C_\pm$ can be a one-sided fold \cite{Tataru1998}.
In this case the general theory of Fourier integral operators implies 
continuity from $H^1$ to $H^{3/4}$ \cite{Greenleaf1994}.
The following example illustrates how singularities in $C_\pm$ arise.

\begin{example}
Suppose $M \subset \R^2$, $c = 1$ identically and that near $(1,0)$ the boundary $\p M$ coincides with the unit circle
\begin{align*}
S^1 := \{x \in \R^2;\ |x| = 1\}.
\end{align*}
If $\K$ intersects the vertical line through $(1,0)$, 
then $C_+$ is not a local canonical graph. 
\end{example}
\begin{proof}
Let us parametrize $S^1$ by $x(a) := (\cos a, \sin a)$, $a \in (-\pi, \pi)$. 
Then the cotangent space $T_{x(a)}^* S^1$ is spanned by 
$x^\perp(a) := (-\sin a, \cos a)$, and
$C_+$ consist locally of the points $(y, \eta, t, a, \tau, \alpha)$ where
\begin{align*}
y = x(a) - t \frac{\eta}{|\eta|},
\quad 
\tau = |\eta|,
\quad 
\alpha = \eta \cdot x^\perp(a).
\end{align*}
Hence $C_+$ is locally parametrized by $\eta$, $t$ and $a$.
Let us consider the projection 
\begin{align*}
\pi_L(t, a, \eta) = (y, \eta), \quad \pi_L : C_+ \to T^*(\K).
\end{align*}
Its differential is of the form 
\begin{align*}
d\pi_L = \ll( \begin{array}{ccc}
- \frac{\eta}{|\eta|} & x^\perp(a) & * \\
0 & 0 & I 
\end{array} \rr),
\end{align*}
where $I$ is the identity matrix and the explicit form of $*$ is not needed for our purposes. 
Thus $d\pi_L (\delta t, \delta a, \delta \eta) = 0$ if and only if $\delta \eta = 0$ and
\begin{align*}
- \frac{\eta}{|\eta|} \delta t + x^\perp(a) \delta a = 0.
\end{align*}
The latter equation has a non-trivial solution if and only if $\eta$ is a scalar multiple of $x^\perp(a)$,
that is, the geodesic $\gamma(t) = y + t \frac{\eta}{|\eta|}$ intersects $S^1$ tangentially. 
In this case $d\pi_L$ has a one dimensional kernel and $C_+$ is not a local canonical graph 
(in fact, $C_+$ is a two-sided fold \cite{Tataru1998}).
\end{proof}

If $\p M$ is strictly convex with respect to the Riemannian metric $c^{-2} dx^2$,
then (\ref{non_tangential}) holds for all $(y_0, \eta_0) \in T^* \K$.
In this case we have the following regularity result.

\begin{theorem} \label{thm_regularity}
Suppose that $\p M$ is strictly convex with respect to the Riemannian metric $c^{-2} dx^2$
and let $\K \subset M^\inter$ be compact.
Let $T > 0$ and let $w$ be the solution of the wave equation
\begin{align*}
&\p_t^2 w - c^2 \Delta w = F,  & \text{in $(0,T) \times \R^n$},
\\\nonumber& w|_{t=0} = 0, \quad \p_t w|_{t=0}= 0,  & \text{in $\R^n$},
\end{align*}
where $F \in C_0^\infty((0,T) \times \K)$.
Then there is $C > 0$ independent of $F$ such that 
\begin{align*}
\norm{w|_{(0,T) \times \p M}}_{H^1((0,T) \times \p M)} \le C \norm{F}_{L^2((0,T) \times \K)}.
\end{align*}
\end{theorem}
\begin{proof}
The operator $F \mapsto w$ coincides with 
the forward parametrix of \cite[Th. 26.1.14]{Hormander1985a} modulo an operator with a smooth kernel.
The composition of the parametrix with the restriction on $(0,T) \times \p M$
is a Fourier integral operator of order $-5/4$ with the canonical relation $C$
consisting of such points 
\begin{align*}
&(t, x, |\xi|, \xi, s, \gamma_{x, \hat \xi}(s-t), |\xi|, \dot\gamma_{x, \hat \xi}^\t(s-t)),
\end{align*}
that $\gamma_{x, \hat \xi}(s-t) \in \p M$, $(x,\xi) \in T^* \K \setminus 0$
and $t, s \in (0,T)$.
Here $\hat \xi = \xi / |\xi|$.

As above, the strict convexity of $\p M$ implies that $s$ is locally a function of 
$(t, x, \hat \xi)$. 
In particular, the canonical relation is parametrized by 
$(t, x, \xi)$ and the projection $C \to T^* ((0,T)\times \K)$ has the rank $2(n+1) - 1$.
We may apply \cite[Th. 4.3.2]{Hormander1971} to get the claimed continuity.
%
\end{proof}

\bibliographystyle{abbrv} 
\bibliography{main}

\end{document}